\documentclass[12p]{amsart}
\usepackage{amssymb}
\usepackage{amsmath}
\usepackage{amsfonts}
\usepackage{geometry}
\usepackage{graphicx}
\usepackage{mathrsfs,amssymb}

\usepackage{hyperref}
\usepackage{cleveref}

\theoremstyle{plain}

\newtheorem{corollary}{Corollary}

\newtheorem{definition}{Definition}

\newtheorem{lemma}{Lemma}

\newtheorem{proposition}{Proposition}
\newtheorem{remark}{Remark}

\newtheorem{theorem}{Theorem}
\numberwithin{equation}{section}

\begin{document}
\title[]{Sequential convergence of a solution to the Chern--Simons--Schr{\"o}dinger equation}

\author{Benjamin Dodson}
\date{\today}

\begin{abstract}
In this paper we prove a sequential convergence result for blowup solutions to the $m$-equivariant, self-dual Chern--Simons--Schr{\"o}dinger equation. We show that if $u$ has mass less than twice the mass of the soliton, a blowup solution converges to the soliton along a subsequence of times that converges to the blowup time.
\end{abstract}
\maketitle

\section{Introduction}
In this paper we prove a sequential convergence result for blowup solutions to the self-dual, $m$-equivariant Chern--Simons--Schr{\"o}dinger equation,
\begin{equation}\label{1.1}
i (\partial_{t} + i A_{t}[u])u + \partial_{r}^{2} u + \frac{1}{r} \partial_{r} u - (\frac{m + A_{\theta}[u]}{r})^{2} u + |u|^{2} u = 0,
\end{equation}
where $m \in \mathbb{Z}$ and $m \geq 1$ is called the equivariance index and the connection components $A_{t}[u]$ and $A_{\theta}[u]$ are given by
\begin{equation}\label{1.2}
A_{t}[u] = -\int_{r}^{\infty} (m + A_{\theta}[u]) |u|^{2} \frac{dr'}{r'}, \qquad A_{\theta}[u] = -\frac{1}{2} \int_{0}^{r} |u|^{2} r' dr'.
\end{equation}
Equation $(\ref{1.1})$ is derived after fixing the Coulomb gauge condition and imposing the equivariant symmetry on the scalar field $\phi$,
\begin{equation}\label{1.3}
\phi(t, x) = u(t, r) e^{im \theta}.
\end{equation}
See \cite{kim2020blow}, \cite{kim2020construction}, and \cite{kim2019pseudoconformal} for more information on this reduction.\medskip

Solutions to $(\ref{1.1})$ have the conserved quantities mass,
\begin{equation}\label{1.4}
M[u(t)] = \int |u(t, x)|^{2} dx,
\end{equation}
and energy,
\begin{equation}\label{1.5}
E[u(t)] = \int \frac{1}{2} |\partial_{r} u|^{2} + \frac{1}{2} (\frac{m + A_{\theta}[u]}{r})^{2} |u|^{2} - \frac{1}{4} |u|^{4} dx.
\end{equation}

We prove that if $u$ is a blowup solution to $(\ref{1.1})$ on the maximal interval $I = (a, b)$, and $u$ blows up forward in time, then $u$ converges weakly to the soliton solution, along a subsequence of times converging to the blowup time.
\begin{theorem}\label{t1.1}
Suppose $u$ is a blowup solution to the Chern--Simons--Schr{\"o}dinger equation, $(\ref{1.1})$ that satisfies
\begin{equation}\label{1.6}
\| u \|_{L^{2}}^{2} < 16 \pi (m + 1).
\end{equation}

Also suppose that $0 \in I = (a, b)$ is the maximal interval of existence of a solution to $(\ref{1.1})$. Furthermore, suppose that the solution blows up forward in time. That is, if $b_{n} \nearrow b$,
\begin{equation}\label{1.7}
\lim_{n \rightarrow \infty} \| u \|_{L_{t,x}^{4}([0, b_{n}] \times \mathbb{R}^{2})} = \infty.
\end{equation}
Then there exists a sequence, $t_{n} \nearrow b$ and $\lambda(t_{n}) > 0$ such that
\begin{equation}\label{1.8}
\lambda(t_{n}) u(t_{n}, \lambda(t_{n}) r) \rightharpoonup Q, \qquad \text{weakly in} \qquad L^{2},
\end{equation}
where $Q$ is the explicit $m$-equivariant static solution.
\end{theorem}

\begin{definition}[Jackiw--Pi vortex]\label{d1.2}
For $m \geq 0$ there is an explicit $m$-equivariant static solution, given by
\begin{equation}\label{1.9}
Q(r) = \sqrt{8} (m + 1) \frac{r^{m}}{1 + r^{2m + 2}}.
\end{equation}
This solution is called the Jackiw--Pi vortex.
\end{definition}
Making a change of variables to

\begin{equation}\label{1.10}
\| Q \|_{L^{2}}^{2} = 2\pi 8(m + 1)^{2} \int_{0}^{\infty} \frac{r^{2m + 1}}{(1 + r^{2m + 2})^{2}} dr, \qquad u = 1 + r^{2m + 2}, \qquad du = 2(m + 1) r^{2m + 1} dr,
\end{equation}
\begin{equation}\label{1.11}
(\ref{1.10}) = 8 \pi (m + 1) \int_{1}^{\infty} \frac{1}{u^{2}} du = 8 \pi (m + 1).
\end{equation}
Therefore, $(\ref{1.3})$ is equivalent to the statement that
\begin{equation}\label{1.12}
\| u \|_{L^{2}}^{2} < 2 \| Q \|_{L^{2}}^{2}.
\end{equation}

Moreover, the Jackiw--Pi vortex solution is the unique minimizer of the energy functional $(\ref{1.5})$. Indeed, observe that the energy functional may be written in the self-dual form 
\begin{equation}\label{1.13}
E[u] = \int \frac{1}{2} |\mathbf{D}_{u} u|^{2},
\end{equation}
where
\begin{equation}\label{1.14}
\mathbf{D}_{u} f = \partial_{r} f - \frac{m + A_{\theta}[u]}{r} f.
\end{equation}
Indeed, expanding $(\ref{1.13})$ using $(\ref{1.14})$,
\begin{equation}\label{1.15}
E[u] = \frac{1}{2} \int |\partial_{r} u|^{2} + \frac{1}{2} \int (\frac{m + A_{\theta}[u]}{r})^{2} |u|^{2} - Re \int (\frac{m + A_{\theta}[u]}{r}) u \overline{\partial_{r} u}.
\end{equation}
Rewriting the last integral in $(\ref{1.15})$ in polar coordinates and integrating by parts in $r$,
\begin{equation}\label{1.16}
\aligned
2 \pi Re \int_{0}^{\infty} (m + A_{\theta}[u]) u (\overline{\partial_{r} u}) dr = \pi Re \int_{0}^{\infty} (m + A_{\theta}[u]) \partial_{r} |u|^{2} dr \\ = -\pi m |u(0)|^{2} - \pi A_{\theta}[u](0) |u(0)|^{2} - \pi \int_{0}^{\infty} \partial_{r} A_{\theta}[u] |u|^{2} dr = \frac{\pi}{2} \int_{0}^{\infty} |u|^{4} r dr = \frac{1}{4} \int |u|^{4} dx.
\endaligned
\end{equation}
\begin{remark}
If $u$ is a Schwarz function with finite energy, then $u(r) \rightarrow 0$ rapidly as $r \rightarrow \infty$ and $u(0) = 0$ if $E[u]$ is finite.
\end{remark}

Furthermore, we can see from Appendix $A$ of \cite{li2022threshold} that $(\ref{1.9})$ is the unique solution to $E[u] = 0$, up to the scaling symmetry of $(\ref{1.1})$,
\begin{equation}\label{1.17}
u(t, r) \mapsto \lambda u(\lambda^{2} t, \lambda r), \qquad \forall \lambda > 0.
\end{equation}
\begin{remark}
A solution to $(\ref{1.1})$ also possesses the phase rotation symmetry,
\begin{equation}\label{1.18}
u(t, r) \mapsto e^{i \gamma} u(t, r), \qquad \gamma \in \mathbb{R}.
\end{equation}
\end{remark}

\subsection{Outline of previous results}
The result of \cite{liu2016global} implies Theorem $\ref{t1.1}$ is vacuously true when $\| u_{0} \|_{L^{2}}^{2} < \| Q \|_{L^{2}}^{2} = 8 \pi (m + 1)$.
\begin{theorem}\label{t1.3}
Suppose $u$ is a solution to $(\ref{1.1})$ with initial data $\| u_{0} \|_{L^{2}}^{2} < 8 \pi (m + 1)$. Then $(\ref{1.1})$ is globally well-posed and scattering in both time directions.
\end{theorem}

In fact, \cite{liu2016global} proved more. For the $m$-equivariant, defocusing Chern--Simons--Schr{\"o}dinger equation, global well--posedness and scattering hold for any data in $L^{2}$. Furthermore, when $m < 0$, the same argument implies that global well-posedness and scattering hold for $(\ref{1.1})$ with initial data in $L^{2}$.\medskip

When $\| u_{0} \|_{L^{2}}^{2} = 8 \pi (m + 1)$, \cite{li2022threshold} proved a rigidity result for $m$-equivariant initial data with finite energy. Let $H_{m}^{1}(\mathbb{R}^{2})$ denote the space of $m$-equivariant functions in $H^{1}(\mathbb{R}^{2})$.
\begin{theorem}[Characterization of the threshold solution in the self-dual case]\label{t1.4}
For $m \geq 1$, $\phi_{0} \in H_{m}^{1}(\mathbb{R}^{2})$, $\| \phi_{0} \|_{L^{2}}^{2} = 8 \pi (m + 1)$, one of the following three scenarios happens:

$(1)$ $u$ is equal to the pseudoconformal transformation of the ground state up to phase rotation and scaling.

$(2)$ $u$ is equal to the ground state up to phase rotation and scaling.

$(3)$ $u$ scatters both forward and backward in time.
\end{theorem}

As in the case of the mass-critical problem, $(\ref{1.1})$ has the pseudoconformal symmetry. Acting on $Q$, the pseudoconformal symmetry gives the blowup solution
\begin{equation}\label{1.19}
S(t, r) = \frac{1}{|t|} Q(\frac{r}{|t|}) e^{-i \frac{r^{2}}{4 |t|}}, \qquad t < 0.
\end{equation}
\begin{remark}
Compare Theorem $\ref{t1.4}$ to \cite{merle1993determination} and \cite{merle1992uniqueness} for the mass--critical nonlinear Schr{\"o}dinger equation.
\end{remark}

Observe that Theorem $\ref{t1.4}$ implies that if $u$ fails to scatter, then $(\ref{1.8})$ holds as $t_{n} \nearrow 0$ for any such sequence. The analogous result for data $\notin H_{m}^{1}$ remains open.\medskip

For data above the ground state that lies in a weighted Sobolev space, \cite{kim2022soliton} proved that the soliton resolution conjecture holds.
\begin{theorem}[Soliton resolution for equivariant $H^{1,1}$ data]\label{t1.4.1}
If $u$ is an $H_{m}^{1}$ solution to $(\ref{1.1})$ that blows up forwards in time at $T < +\infty$, then $u(t)$ admits the decomposition
\begin{equation}\label{1.20}
u(t, \cdot) - \lambda(t) e^{i \gamma(t)} Q(\lambda(t) \cdot) \rightarrow z^{\ast} \in L^{2}, \qquad \text{as} \qquad t \nearrow T,
\end{equation}
for some function $\lambda(t) : [0, T) \rightarrow (0, \infty)$ and $\gamma(t) : [0, T) \rightarrow \mathbb{R}$.\medskip

If $u$ is a $H_{m}^{1, 1}$ solution to $(\ref{1.1})$ that exists globally forwards in time, then either $u(t)$ scatters forward in time, or $u(t)$ admits the decomposition
\begin{equation}\label{1.21}
u(t, \cdot) - \lambda(t) e^{i \gamma(t)} Q(\lambda(t) \cdot) - e^{it \Delta^{(-m - 2)}} u^{\ast} \rightarrow 0, \qquad \text{in} \qquad L^{2}, \qquad \text{as} \qquad t \rightarrow +\infty,
\end{equation}
for some $u^{\ast} \in L^{2}$.
\end{theorem}

\begin{remark}
The operator $e^{it \Delta^{(-m - 2)}}$ denotes the linear evolution of a $(-m - 2)$ equivariance index Chern--Simons--Schr{\"o}dinger operator. For our purposes, it is enough to know that
\begin{equation}\label{1.22}
e^{it \Delta^{(-m - 2)}} u^{\ast} \rightharpoonup 0.
\end{equation}
\end{remark}

The papers of \cite{kim2020blow}, \cite{kim2020construction}, and \cite{kim2019pseudoconformal} construct blowup solutions for a class of data with mass slightly higher than the mass of the ground state. In each case, the blowup solutions satisfy $(\ref{1.8})$.\medskip

See also the result of \cite{dodson2023liouville} for a rigidity result for data above the ground state, namely that
\begin{theorem}\label{t1.5}
Suppose $u_{0} \in H_{m}^{1}$ is an initial data for $(\ref{1.1})$ that has a solution on the maximal interval of existence $I$. Furthermore, suppose that $I = (-\infty, t_{0})$, where $t_{0}$ could be $+\infty$, or $(t_{0}, \infty)$, where $t_{0}$ could be $-\infty$. Also, suppose that for any $\eta > 0$, there exists $R(\eta) < \infty$ such that
\begin{equation}\label{1.23}
\sup_{t \in I} \int_{|x| \geq R(\eta)} |u(t, x)|^{2} dx < \eta,
\end{equation}
where $I$ is the interval of existence for a solution to $(\ref{1.1})$. Then $u$ is the soliton solution up to the symmetries $(\ref{1.17})$ and $(\ref{1.18})$.
\end{theorem}

\begin{remark}
Compare Theorem $\ref{t1.1}$ to the results of \cite{fan20182}, \cite{dodson20212}, and \cite{dodson2022l2} for the mass-critical nonlinear Schr{\"o}dinger equation.
\end{remark}

\subsection{Outline of argument}
The proof of Theorem $\ref{t1.1}$ may be broken into a number of steps.\medskip

We begin with a blowup solution at the threshold, $\| u \|_{L^{2}}^{2} = 8 \pi (m + 1)$. Such a solution need not have finite energy. Using \cite{liu2016global}, a solution to $(\ref{1.1})$ with mass equal to $8 \pi (m + 1)$ is a minimal mass blowup solution. Therefore, using the profile decomposition in \cite{liu2016global}, if $t_{n}$ converges to the blowup time, then $u(t_{n})$ converges to the initial data of a solution that blows up forward and backward in time. Therefore, $u$ is an almost periodic solution.\medskip

We then prove that if $u$ is an almost periodic blowup solution (of any mass), then it is possible to take a sequence $t_{n}$ converging to the blowup time, and such that, after rescaling, $u$ satisfies one of three scenarios:

\begin{enumerate}
\item $N(t) = 1$, $t \in \mathbb{R}$,

\item $N(t) = t^{-1/2}$, $t \in (0, \infty)$,

\item $N(t) \leq 1$, $t \in \mathbb{R}$, $\liminf_{t \rightarrow \pm \infty} N(t) = 0$.
\end{enumerate}

Following \cite{liu2016global} (and also \cite{dodson2023liouville}), we prove that the only such solution is the soliton.\medskip

Finally, suppose that $u$ is a blowup solution with mass in the interval
\begin{equation}\label{1.24}
\| Q \|_{L^{2}}^{2} < \| u \|_{L^{2}}^{2} < 2 \| Q \|_{L^{2}}^{2}.
\end{equation}
Suppose $u$ blows up forward in time. Taking the profile decomposition of \cite{liu2016global}, for a sequence $t_{n}$ converging to the blowup time, one profile must be the initial data for a solution to $(\ref{1.1})$ that blows up both forward and backward in time. Indeed, by the decoupling property, at most one profile can have mass greater than or equal to the mass of the soliton, and therefore, all other profiles must scatter forward and backward in time. Furthermore, if we solve $(\ref{1.1})$ backward in time, with initial data $u(t_{n}) = u_{0}^{(n)}$, we must have
\begin{equation}\label{1.25}
\| u^{(n)} \|_{(\inf(I), 0]} \nearrow \infty,
\end{equation}
so one profile must blowup both forward and backward in time. Minimizing over all such solutions which can be obtained via a profile decomposition of a sequence gives an almost periodic blowup solution, and therefore, after passing to a subsequence, must be $Q$ (up to symmetries $(\ref{1.17})$ and $(\ref{1.18})$).

\section{Almost periodic solutions}
We begin with almost periodic solutions.
\begin{definition}[Almost periodic solution]\label{d2.1}
A solution $u$ to $(\ref{1.1})$ with lifespan $I$ is said to be almost periodic modulo scaling if there exists a frequency scale function $N : I \rightarrow \mathbb{R}^{+}$ and a compactness modulus function $C : \mathbb{R}^{+} \rightarrow \mathbb{R}^{+}$ such that
\begin{equation}\label{2.1}
\int_{|x| \geq \frac{C(\eta)}{N(t)}} |u(t, x)|^{2} dx + \int_{|\xi| \leq C(\eta) N(t)} |\hat{u}(t, \xi)|^{2} d\xi \leq \eta,
\end{equation}
for all $t \in I$ and $\eta > 0$.
\end{definition}

\begin{remark}
This definition of almost periodicity is identical to the definition of almost periodicity for the mass--critical nonlinear Schr{\"o}dinger equation. See \cite{dodson2019defocusing} and the references therein for more information.
\end{remark}

We prove that an almost periodic solutions with certain additional properties must be the soliton.
\begin{theorem}\label{t2.2}
Suppose $u$ is an almost periodic solution to $(\ref{1.1})$ with finite, nonzero mass. Let $I$ be the maximal interval of existence. Suppose further that $N(t)$ satisfies one of three conditions:

\begin{enumerate}
\item $I = \mathbb{R}$ and $N(t) = 1$,

\item $I = \mathbb{R}$, $N(t) \leq 1$ for all $t \in \mathbb{R}$, and $\liminf_{t \rightarrow \pm \infty} N(t) = 0$,

\item $I = (0, \infty)$ and $N(t) = t^{-1/2}$.
\end{enumerate}
Then $u$ is a soliton solution. That is, there exists some $\gamma \in \mathbb{R}$ and $\lambda > 0$ such that, for all $t \in \mathbb{R}$,
\begin{equation}\label{2.2}
u(t, r) = e^{i \gamma} \lambda Q(\lambda r).
\end{equation}

\end{theorem}

Observe that if $(\ref{2.2})$ holds, then $N(t) = 1$. Thus, a consequence of Theorem $\ref{t2.2}$ is that $(2)$ and $(3)$ cannot occur for an almost periodic solution to $(\ref{1.1})$.
\begin{proof}
The proof of Theorem $\ref{t2.2}$ utilizes the extra regularity argument in \cite{liu2016global}, see Theorems $5.3$ and $5.13$ in \cite{liu2016global}.
\begin{theorem}\label{t2.3}
In cases $(1)$ and $(2)$ in Theorem $\ref{t2.2}$, for any $s > 0$,
\begin{equation}\label{2.3}
\| u \|_{L_{t}^{\infty} H_{x}^{s}(\mathbb{R} \times \mathbb{R}^{2})} \lesssim_{s, u} 1.
\end{equation}
In case $(3)$ of Theorem $\ref{t2.2}$, for any $s > 0$, $t \in (0, \infty)$,
\begin{equation}\label{2.4}
\| u(t) \|_{\dot{H}_{x}^{s}(\mathbb{R}^{2})} \lesssim_{u, s} t^{-s/2}.
\end{equation}
\end{theorem}

Now then, taking $s = 1$ in case $(1)$, Theorem $\ref{t1.5}$ combined with $(\ref{2.1})$ implies that $u$ must be a soliton solution.

In the second case, we show that $E(u(t)) = 0$. Let $t_{n} \rightarrow \infty$ be a sequence such that $N(t_{n}) \searrow 0$ as $n \rightarrow \infty$. 
\begin{equation}\label{2.5}
\lim_{n \rightarrow \infty} \| P_{\leq N(t_{n})^{1/2}} u(t_{n}) \|_{\dot{H}^{1}} = 0.
\end{equation}
Furthermore, since $N(t_{n}) \searrow 0$, $N(t_{n})^{1/2} \geq C(\eta) N(t_{n})$ for any $\eta > 0$ if $n$ is sufficiently large. Therefore interpolating $(\ref{2.1})$ with $(\ref{2.3})$ with $s = 2$ implies
\begin{equation}\label{2.6}
\lim_{n \rightarrow \infty} \| P_{> N(t_{n})^{1/2}} u(t_{n}) \|_{\dot{H}^{1}} = 0.
\end{equation}
Then by the Sobolev embedding theorem, $(\ref{2.6})$, the fact that $u(t, 0) = 0$, Hardy's inequality, and conservation of energy, we have $E(u(t)) = 0$. Therefore, by $(\ref{1.18})$, $u$ is a soliton.
\begin{remark}
To see why we can use Hardy's inequality, observe that $(\ref{2.3})$ combined with the Sobolev embedding theorem implies that
\begin{equation}\label{2.7}
\sup_{t \in \mathbb{R}, r > 0} |\partial_{r} u(t, r)| \lesssim_{u} 1,
\end{equation}
which by the fundamental theorem of calculus and $u(t, 0) = 0$ implies
\begin{equation}\label{2.8}
|u(t, r)| \lesssim_{u} r.
\end{equation}
\end{remark}
Applying Hardy's inequality on the annulus $\{ x : \frac{1}{R} \leq |x| \leq R \}$ combined with $(\ref{2.8})$ and conservation of mass,
\begin{equation}\label{2.9}
\int \frac{1}{r^{2}} |u(t, r)|^{2} dr \lesssim \frac{1}{R^{2}} \| u \|_{L^{2}}^{2} + \ln(R) \| \nabla u(t) \|_{L^{2}}^{2}.
\end{equation}
Taking $R \nearrow \infty$ sufficiently slowly as $\| \nabla u(t_{n}) \|_{L^{2}} \searrow 0$ implies
\begin{equation}\label{2.10}
\int \frac{1}{2} (\frac{m + A_{\theta}[u]}{r})^{2} |u(t_{n})|^{2} \searrow 0,
\end{equation}
as $n \rightarrow \infty$.

By a similar computation, in case $(3)$, we have $E(u(t)) \searrow 0$ as $t \rightarrow \infty$. Once again, conservation of energy implies that $E(u(t)) = 0$, and therefore that $u(t)$ is a soliton solution. This proves Theorem $\ref{t2.2}$.
\end{proof}

Next, following section three of \cite{liu2016global}, we prove that if $u$ is an almost periodic solution, $u$ must converge to the soliton along a subsequence. Recall the local constancy property, Proposition $3.8$ from \cite{liu2016global},
\begin{proposition}\label{p2.3}
Let $u$ be a nonzero, almost periodic solution to $(\ref{1.1})$. Then there exists $\delta(u) > 0$ such that for every $t_{0} \in I$,
\begin{equation}\label{2.11}
[t_{0} - \delta N(t_{0})^{-2}, t_{0} + \delta N(t_{0})^{-2}] \subset I,
\end{equation}
and
\begin{equation}\label{2.12}
N(t) \sim_{u} N(t_{0}), \qquad \text{whenever} \qquad |t - t_{0}| \leq \delta N(t_{0})^{-2}.
\end{equation}
\end{proposition}

For any $T > 0$, define the quantities,
\begin{equation}\label{2.13}
osc(T) = \inf_{t_{0} \in I} \frac{\sup_{t \in I, |t - t_{0}| \leq T N(t_{0})^{-2}} N(t)}{\inf_{t \in I, |t - t_{0}| \leq T N(t_{0})^{-2}} N(t)},
\end{equation}
and
\begin{equation}\label{2.14}
a(t_{0}) = \frac{\inf_{t \in I, t \leq t_{0}} N(t) + \inf_{t \in I, t \geq t_{0}} N(t)}{N(t_{0})}.
\end{equation}

If $\lim_{T \rightarrow \infty} osc(T) < \infty$, then $(\ref{2.11})$ and $(\ref{2.12})$ imply that there exists a sequence of times $t_{n} \in I$, $T_{n} \rightarrow \infty$, such that
\begin{equation}\label{2.15}
[t_{n} - T_{n} N(t_{n})^{-2}, t_{n} + T_{n} N(t_{n})^{-2}] \subset I.
\end{equation}
By the Arzela--Ascoli theorem, $(\ref{2.1})$ implies that, possibly after passing to a subsequence, there exists some $u_{0} \in L^{2}$ such that
\begin{equation}\label{2.16}
N(t_{n})^{-1} e^{i \gamma(t_{n})} u(t_{n}, N(t_{n})^{-1} x) \rightarrow u_{0} \in L^{2}.
\end{equation}
Furthermore, $(\ref{2.15})$ implies that $u_{0}$ is the initial data to a global, almost periodic solution to $(\ref{1.1})$ that satisfies
\begin{equation}\label{2.17}
0 < \inf_{t \in \mathbb{R}} N(t) \leq \sup_{t \in \mathbb{R}} N(t) < \infty.
\end{equation}
After modifying $C(\eta)$ in $(\ref{2.1})$ by a constant, $u_{0}$ is the initial data for an almost periodic, blowup solution to $(\ref{1.1})$ that satisfies $(1)$ in Theorem $\ref{t2.2}$. Theorem $\ref{t2.2}$ also implies that $(\ref{2.16})$ holds for
\begin{equation}\label{2.18}
u_{0} = \lambda e^{i \gamma} Q(\lambda r),
\end{equation}
for some $\lambda > 0$ and $\gamma \in \mathbb{R}$. Taking $\lambda(t_{n}) = N(t_{n})^{-1}$, we are done.\medskip

Therefore, the proof of
\begin{theorem}\label{t2.4}
If $u$ is an almost periodic solution to $(\ref{1.1})$ on a maximal interval $I$, then there exists a sequence $t_{n} \in I$ and a sequence $\lambda(t_{n}) > 0$, $\gamma(t_{n}) \in \mathbb{R}$, such that
\begin{equation}\label{2.19}
\lambda(t_{n}) e^{i \gamma(t_{n})} u(t_{n}, \lambda(t_{n}) \cdot) \rightarrow Q, \qquad \text{in} \qquad L^{2}.
\end{equation}
\end{theorem}
reduces to proving
\begin{lemma}\label{l2.5}
There does not exist an almost periodic solution to $(\ref{1.1})$ that satisfies
\begin{equation}\label{2.20}
\lim_{T \rightarrow \infty} osc(T) = \infty.
\end{equation}
\end{lemma}
\begin{proof}[Proof of Lemma $\ref{l2.5}$]
First suppose $\lim_{T \rightarrow \infty} osc(T) = \infty$ and that $\inf_{t_{0} \in I} a(t_{0}) = 0$. In this case, there exists a sequence $t_{n} \in I$ and $t_{n}^{-} < t_{n} < t_{n}^{+}$ such that
\begin{equation}\label{2.21}
\frac{N(t_{n}^{-})}{N(t_{n})} \rightarrow 0, \qquad \text{and} \qquad \frac{N(t_{n}^{+})}{N(t_{n})} \rightarrow 0.
\end{equation}
Now choose $t_{n}'$ satisfying $t_{n}^{-} < t_{n}' < t_{n}^{+}$ and such that
\begin{equation}\label{2.22}
N(t_{n}') \sim \sup_{t_{n}^{-} < t < t_{n}^{+}} N(t).
\end{equation}
Again, by the Arzela--Ascoli theorem, possibly after passing to a subsequence,
\begin{equation}\label{2.23}
N(t_{n}')^{-1} e^{i \gamma(t_{n}')} u(t_{n}', N(t_{n}')^{-1} \cdot) \rightarrow u_{0}, \qquad \text{in} \qquad L^{2}.
\end{equation}
Now by $(\ref{2.11})$ and $(\ref{2.12})$,
\begin{equation}\label{2.24}
\int_{t_{n}^{-}}^{t_{n}'} N(t)^{-2} dt \rightarrow \infty, \qquad \text{and} \qquad \int_{t_{n}'}^{t_{n}^{+}} N(t)^{-2} dt \rightarrow \infty.
\end{equation}
Therefore, $u_{0}$ is the initial data for an almost periodic solution to $(\ref{1.1})$ that satisfies $N(t) \lesssim 1$ for all $t \in \mathbb{R}$. Also, by definition of $osc(T)$, $(\ref{2.20})$ must also hold for this solution, so in particular, this solution must satisfy $(2)$ of Theorem $\ref{t2.2}$. But, as has already been mentioned, such a solution cannot exist, so therefore, an almost periodic solution satisfying $osc(T) \rightarrow \infty$ and $\inf_{t_{0} \in I} a(t_{0}) = 0$ cannot exist.\medskip

In the case that $osc(T) \rightarrow \infty$ and $\inf_{t_{0} \in I} a(t_{0}) > 0$, it is possible to follow the arguments in \cite{killip2009cubic} to prove that, possibly after passing to a sequence and rescaling,
\begin{equation}\label{2.25}
N(t_{n})^{-1} e^{i \gamma(t_{n})} u(t_{n}, N(t_{n})^{-1} \cdot) \rightarrow u_{0}, \qquad \text{in} \qquad L^{2},
\end{equation}
and that $u_{0}$ is the initial data for an almost periodic solution satisfying $(3)$ in Theorem $\ref{t2.2}$. Again by Theorem $\ref{t2.2}$, such a solution does not exist, which proves Lemma $\ref{l2.5}$.
\end{proof}

It also follows directly from Theorem $\ref{t2.4}$ that
\begin{corollary}\label{c2.6}
If $u$ is an almost periodic solution to $(\ref{1.1})$, then $\| u \|_{L^{2}} = \| Q \|_{L^{2}}$.
\end{corollary}

\section{Solutions with mass equal to the ground state}
Now, we prove Theorem $\ref{t1.1}$ in the case when $\| u_{0} \|_{L^{2}}^{2} = 8 \pi (m + 1)$, the mass of the ground state.
\begin{theorem}\label{t3.1}
Suppose $u(t, r)$ is a solution to $(\ref{1.1})$ that blows up forward in time on the maximal interval in time, $I = (a, b)$. Then, there exists a sequence $t_{n} \nearrow b$ and $\lambda(t_{n}) > 0$, such that
\begin{equation}\label{3.1}
\lambda(t_{n}) u(t_{n}, \lambda(t_{n}) r) \rightarrow Q, \qquad \text{in} \qquad L^{2}.
\end{equation}
\end{theorem}
\begin{remark}
Note that in this case, as in the previous case, we have convergence in $L^{2}$ norm, not just weak convergence in $L^{2}$.
\end{remark}
\begin{proof}
First use the profile decomposition in \cite{liu2016global}.
\begin{proposition}[Linear profile decomposition]\label{p3.2}
Let $G_{max}$ be the symmetry group generated by scaling transformations,
\begin{equation}\label{3.2}
g_{\lambda} f(r) = \lambda^{-1} f(\lambda^{-1} r).
\end{equation}
Let $\psi_{n}$ be a bounded sequence in $L_{m}^{2}$. Then, possibly after passing to a subsequence, there exists a sequence of functions $\phi^{j} \in L_{m}^{2}$, group elements $g_{n}^{j} \in G_{max}$, and times $t_{n}^{j} \in \mathbb{R}$, such that we have the decomposition
\begin{equation}\label{3.3}
\psi_{n} = \sum_{j = 1}^{J} g_{n}^{j} e^{it_{n}^{j} \Delta} \phi^{j} + w_{n}^{J},
\end{equation}
for all $j = 1, 2, ...$. Moreover, $w_{n}^{J} \in L_{m}^{2}$ is such that its linear evolution has asymptotically vanishing scattering size
\begin{equation}\label{3.4}
\lim_{J \rightarrow \infty} \limsup_{n \rightarrow \infty} \| e^{it \Delta} w_{n}^{J} \|_{L_{t,x}^{4}} = 0.
\end{equation}
Moreover, for any $j \neq j'$,
\begin{equation}\label{3.5}
\frac{\lambda_{n}^{j}}{\lambda_{n}^{j'}} + \frac{\lambda_{n}^{j'}}{\lambda_{n}^{j}} + \frac{|t_{n}^{j} (\lambda_{n}^{j})^{2} - t_{n}^{j'} (\lambda_{n}^{j'})^{2}|}{\lambda_{n}^{j} \lambda_{n}^{j'}} \rightarrow \infty.
\end{equation}
Furthermore, for any $J \geq 1$, we have the mass decoupling property,
\begin{equation}\label{3.6}
\lim_{n \rightarrow \infty} [\| \psi_{n} \|_{L^{2}}^{2} - \sum_{j = 1}^{J} \| \psi_{j} \|_{L^{2}}^{2} - \| w_{n}^{J} \|_{L^{2}}^{2}] = 0.
\end{equation}
\end{proposition}
The proof of profile decomposition is the same as in the mass-critical nonlinear Schr{\"o}dinger equation, see \cite{MR2445122}.\medskip

Now suppose that $u$ is a solution to $(\ref{1.1})$ that blows up forward in time. That is, for any $t_{0} \in I$, where $I$ is the maximal interval of existence of $u$, then
\begin{equation}\label{3.7}
\| u \|_{L_{t,x}^{4}([t_{0}, \sup(I)) \times \mathbb{R}^{2})} = \infty.
\end{equation}
Let $t_{n} \nearrow \sup(I)$ be a sequence of times converging to $\sup(I)$. Applying Theorem $\ref{t3.1}$ to $u(t_{n}, r)$, possibly after passing to a subsequence,
\begin{equation}\label{3.8}
u(t_{n}) = \sum_{j = 1}^{J} g_{n}^{j} e^{it_{n}^{j} \Delta} \phi^{j} + w_{n}^{J}.
\end{equation}
\begin{lemma}\label{l3.3}
After relabeling, $\| \phi^{1} \|_{L^{2}} = \| Q \|_{L^{2}}$.
\end{lemma}
\begin{proof}
Suppose instead that
\begin{equation}\label{3.9}
\sup_{j} \| \phi^{j} \|_{L^{2}} < \| Q \|_{L^{2}}.
\end{equation}
For any $j$, let $v_{n}^{j}(t)$ be the solution to $(\ref{1.1})$ that satisfies
\begin{equation}\label{3.10}
v_{n}^{j}(t_{n}) = g_{n}^{j} e^{i t_{n}^{j} \Delta} \phi^{j}.
\end{equation}
By Theorem $\ref{t1.3}$,
\begin{equation}\label{3.11}
\sup_{j} \| v_{n}^{j} \|_{L_{t,x}^{4}(\mathbb{R} \times \mathbb{R}^{2})} < \infty.
\end{equation}
Also, rearranging $\phi^{j}$ so that $\| \phi^{j} \|_{L^{2}}$ is decreasing, there exists some $J_{0} < \infty$ such that, for any $j > J_{0}$,
\begin{equation}\label{3.12}
\| v_{n}^{j} \|_{L_{t,x}^{4}(\mathbb{R} \times \mathbb{R}^{2})} \lesssim \| v_{n}^{j} \|_{L_{t}^{\infty} L_{x}^{2}}.
\end{equation}
Therefore, by $(\ref{3.5})$ and $(\ref{3.6})$,
\begin{equation}\label{3.13}
\sup_{J} \lim_{n \rightarrow \infty} \| \sum_{j = 1}^{J} v_{n}^{j}(t) \|_{L_{t,x}^{4}(\mathbb{R} \times \mathbb{R}^{2})} < \infty.
\end{equation}
Furthermore, by standard perturbation arguments, see for example section three of \cite{liu2016global}, $(\ref{3.5})$ implies that in this case, for any $J < \infty$
\begin{equation}\label{3.14}
\tilde{u}_{n}^{J}(t) \rightarrow \sum_{j = 1}^{J} v_{n}^{j}(t), \qquad \text{in} \qquad L_{t,x}^{4}, \qquad \text{as} \qquad n \rightarrow \infty,
\end{equation}
where $\tilde{u}_{n}^{J}(t)$ is the solution to $(\ref{1.1})$ with
\begin{equation}\label{3.15}
\tilde{u}_{n}^{J}(t_{n}) = \sum_{j = 1}^{J} g_{n}^{j} e^{it_{n}^{j} \Delta} \phi^{j}.
\end{equation}
Combining $(\ref{3.4})$, $(\ref{3.13})$, and standard perturbation arguments, we obtain a contradiction if $(\ref{3.9})$ holds.
\end{proof}

Lemma $\ref{l3.3}$ implies that $\phi^{j} = 0$ for $j > 1$, and by $(\ref{3.6})$,
\begin{equation}\label{3.16}
\lim_{n \rightarrow \infty} \| w_{n} \|_{L^{2}} = 0.
\end{equation}
Furthermore, we must have $\sup_{n} |t_{n}^{1}| < \infty$. Indeed, by $(\ref{3.7})$ and $t_{n} \nearrow \sup(I)$,
\begin{equation}\label{3.17}
\lim_{n \rightarrow \infty} \| u \|_{L_{t,x}^{4}((\inf(I), t_{n}] \times \mathbb{R}^{2})} = \infty, \qquad \text{and} \qquad \lim_{n \rightarrow \infty} \| u \|_{L_{t,x}^{4}([t_{n}, \sup(I)) \times \mathbb{R}^{2})} = \infty.
\end{equation}
On the other hand, if, after passing to a subsequence, $t_{n}^{1} \nearrow \infty$, Strichartz estimates (see \cite{strichartz1977restrictions}) imply
\begin{equation}\label{3.18}
\| e^{it \Delta} g_{n}^{j} e^{it_{n}^{1} \Delta} \phi^{j} \|_{L_{t,x}^{4}([0, \infty) \times \mathbb{R}^{2})} = 0,
\end{equation}
which by standard perturbation arguments combined with $(\ref{3.16})$ would contradict $(\ref{3.17})$ forward in time. A similar argument rules out $t_{n}^{1} \searrow -\infty$. Therefore, passing to a further subsequence, we can take $t_{n}^{1} \rightarrow t_{0}$, and then absorb $t_{0}$ into $\phi^{1}$,
\begin{equation}\label{3.19}
e^{it_{n}^{1} \Delta} \phi^{1} \mapsto \phi^{1}.
\end{equation}
Therefore, we have proved that if $t_{n} \nearrow \sup(I)$, then, possibly after passing to a subsequence, there exists $\lambda(t_{n}) > 0$ and $\gamma(t_{n}) \in \mathbb{R}$ such that
\begin{equation}\label{3.20}
\lambda(t_{n}) u(t_{n}, \lambda(t_{n}) \cdot) \rightarrow \phi, \qquad \text{in} \qquad L_{m}^{2}.
\end{equation}

Now let $\tilde{u}(t, r)$ be the solution to $(\ref{1.1})$ with initial data equal to $\phi$. Again, $(\ref{3.17})$ implies that if $\tilde{I}$ is the maximal interval of existence of $\tilde{u}$,
\begin{equation}\label{3.21}
\| \tilde{u} \|_{L_{t,x}^{4}((\inf(\tilde{I}), 0] \times \mathbb{R}^{2})} = \| \tilde{u} \|_{L_{t,x}^{4}([0, \sup(\tilde{I})) \times \mathbb{R}^{2})} = \infty.
\end{equation}
Applying the above arguments to a sequence $\tilde{t}_{n} \in \tilde{I}$; this time $\tilde{t}_{n}$ can go to $\sup(\tilde{I})$, $\inf(\tilde{I})$, or converge along a subsequence to some $t_{0} \in \tilde{I}$; after passing to a subsequence, there exists some $\phi \in L_{m}^{2}$ such that
\begin{equation}\label{3.22}
\lambda(\tilde{t}_{n}) u(\tilde{t}_{n}, \lambda(\tilde{t}_{n}) \cdot) \rightarrow \phi, \qquad \text{in} \qquad L_{m}^{2}.
\end{equation}
Thus, $u(t)$ is sequentially compact in $L^{2} / G_{max}$, which by the Arzela--Ascoli theorem is equivalent to $(\ref{2.1})$. Therefore, we are back in section two.\medskip

By Theorem $\ref{t2.4}$, there exists a sequence $\tilde{t}_{n} \in \tilde{I}$, $\tilde{\lambda}(\tilde{t}_{n})$, and $\gamma(\tilde{t}_{n}) \in \mathbb{R}$, such that
\begin{equation}\label{3.23}
\tilde{\lambda}(\tilde{t}_{n}) e^{i \gamma(\tilde{t}_{n})} \tilde{u}(\tilde{t}_{n}, \tilde{\lambda}(\tilde{t}_{n}) \cdot) \rightarrow Q, \qquad \text{in} \qquad L_{m}^{2}.
\end{equation}
Now, by $(\ref{3.20})$,
\begin{equation}\label{3.24}
\lambda(t_{n}) u(t_{n}, \lambda(t_{n}) \cdot) \rightarrow \tilde{u}(0), \qquad \text{in} \qquad L_{m}^{2},
\end{equation}
and again, by standard perturbative arguments and $(\ref{1.17})$, for any $t \in \tilde{I}$, for $n' \geq N(t)$ sufficiently large,
\begin{equation}\label{3.25}
\lambda(t_{n'}) u(t_{n'} + \lambda(t_{n'})^{2} t, \lambda(t_{n'}) \cdot) \rightarrow \tilde{u}(t), \qquad \text{in} \qquad L_{m}^{2}.
\end{equation}
\begin{remark}
This is not completely straightforward since $t_{n'} + \lambda(t_{n'})^{2} t$ might not lie in the interval of existence of $u$. Nevertheless, since we can have at most one blowup profile, standard perturbative analysis implies that $t_{n'} + \lambda(t_{n'})^{2} t \in I$ for $n'$ sufficiently large.
\end{remark}

Therefore, for any fixed $n$,
\begin{equation}\label{3.26}
\lambda(t_{n'}) u(t_{n'} + \lambda(t_{n'})^{2} \tilde{t}_{n}, \lambda(t_{n'}) \cdot) \rightarrow \tilde{u}(\tilde{t}_{n}), \qquad \text{in} \qquad L_{m}^{2},
\end{equation}
as $n' \rightarrow \infty$, and therefore,
\begin{equation}\label{3.27}
\lambda(t_{n'}) \tilde{\lambda}(\tilde{t}_{n}) u(t_{n'} + \lambda(t_{n'})^{2} \tilde{t}_{n}, \tilde{\lambda}(\tilde{t}_{n}) \lambda(t_{n'}) \cdot) \rightarrow \tilde{\lambda}(\tilde{t}_{n}) \tilde{u}(\tilde{t}_{n}, \tilde{\lambda}(\tilde{t}_{n}) \cdot), \qquad \text{in} \qquad L_{m}^{2},
\end{equation}
as $n' \rightarrow \infty$. Therefore, we can choose $\epsilon = 2^{-n}$ such that
\begin{equation}\label{3.28}
\| \lambda(t_{n'(n)}) \tilde{\lambda}(\tilde{t}_{n}) u(t_{n'(n)} + \lambda(t_{n'(n)})^{2} \tilde{t}_{n}, \tilde{\lambda}(\tilde{t}_{n}) \lambda(t_{n'(n)}) \cdot) - \tilde{\lambda}(\tilde{t}_{n}) \tilde{u}(\tilde{t}_{n}, \tilde{\lambda}(\tilde{t}_{n}) \cdot) \|_{L^{2}} \leq 2^{-n}.
\end{equation}
Combining $(\ref{3.23})$ and $(\ref{3.28})$,
\begin{equation}\label{3.29}
\lambda(t_{n'(n)}) \tilde{\lambda}(\tilde{t}_{n}) u(t_{n'(n)} + \lambda(t_{n'(n)})^{2} \tilde{t}_{n}, \tilde{\lambda}(\tilde{t}_{n}) \lambda(t_{n'(n)}) \cdot) \rightarrow Q, \qquad \text{in} \qquad L_{m}^{2}.
\end{equation}
This proves Theorem $\ref{t3.1}$.
\end{proof}

\section{Sequential convergence}
Now we are ready to prove Theorem $\ref{t1.1}$. Suppose $u$ is a solution to $(\ref{1.1})$ that satisfies
\begin{equation}\label{4.1}
8 \pi(m + 1) < \| u \|_{L^{2}}^{2} < 16 \pi (m + 1),
\end{equation}
and $u$ blows up forward in time. That is, for some $t_{0} \in I$, where $I$ is the interval of existence of $u$,
\begin{equation}\label{4.2}
\| u \|_{L_{t,x}^{4}([t_{0}, \sup(I)) \times \mathbb{R}^{2})} = \infty.
\end{equation}
Now, define the set,
\begin{definition}\label{d4.1}
\begin{equation}\label{4.3}
\mathcal M_{u} = \{ \| u_{0} \|_{L^{2}}^{2} : u_{0} \text{ is the initial data for a blowup solution to } (\ref{1.1}) \} \subset [8 \pi(m + 1), 16 \pi(m + 1)),
\end{equation}
and furthermore there exists $t_{n} \nearrow \sup(I)$, $\lambda(t_{n}) > 0$, such that
\begin{equation}\label{4.4}
\lambda(t_{n}) u(t_{n}, \lambda(t_{n}) \cdot) \rightharpoonup u_{0}, \qquad \text{weakly in} \qquad L_{m}^{2}.
\end{equation}
\end{definition}
\begin{lemma}\label{l4.1}
The set $\mathcal M_{u}$ contains its infimum. That is, if
\begin{equation}\label{4.5}
m = \inf \mathcal M_{u},
\end{equation}
then there exists some $u_{0} \in L_{m}^{2}$, $\| u_{0} \|_{L^{2}}^{2} = m$, $u_{0}$ is the initial data for a blowup solution to $(\ref{1.1})$, and there exists a sequence $t_{n} \nearrow \sup(I)$ and $\lambda(t_{n}) > 0$ such that $(\ref{4.4})$ holds.
\end{lemma}
\begin{proof}
Let $\| u_{0}^{(n')} \|_{L^{2}}^{2} \searrow \inf \mathcal M_{u}$ be a sequence of masses where $u_{0}^{(n')} \in L_{m}^{2}$ are the initial data for blowup solutions to $(\ref{1.1})$, and furthermore, for any $n'$, there exists a sequence $t_{n(n')} \nearrow \sup(I)$ and $\lambda(t_{n(n')}) > 0$ such that
\begin{equation}\label{4.6}
\lambda(t_{n(n')}) u(t_{n(n')}, \lambda(t_{n(n')}) \cdot) \rightharpoonup u_{0}^{(n')}, \qquad \text{weakly in} \qquad L_{m}^{2}.
\end{equation}
Apply Proposition $\ref{p3.2}$ to $u_{0}^{(n')}$. Then for any $J$,
\begin{equation}\label{4.7}
u_{0}^{(n')} = \sum_{j = 1}^{J} g_{n'}^{j} e^{it_{n'}^{j} \Delta} \phi^{j} + w_{n'}^{J},
\end{equation}
where $(\ref{3.4})$--$(\ref{3.6})$ holds.\medskip

\noindent \textbf{Claim:}
\begin{equation}\label{4.8}
\sup_{j} \| \phi^{j} \|_{L^{2}}^{2} \geq 8 \pi(m + 1).
\end{equation}
Otherwise, Theorem $\ref{t1.3}$ combined with the computations in the proof of Lemma $\ref{l3.3}$ would imply scattering for $n'$ sufficiently large. Relabelling, suppose that $\| \phi^{1} \|_{L^{2}}^{2} \geq 8 \pi(m + 1)$. The mass decoupling property, $(\ref{3.6})$, implies that $\| \phi^{j} \|_{L^{2}}^{2} < 8 \pi(m + 1)$ for any $j \geq 2$.\medskip

\noindent \textbf{Claim:} Next, we show that, possibly after passing to a subsequence, $t_{n'}^{1}$ converges in $\mathbb{R}$. This follows from the fact that if $u_{0}^{(n')}$ is the initial data for a solution $u^{(n')}$ to $(\ref{1.1})$, $(\ref{4.6})$ implies
\begin{equation}\label{4.9}
\| u^{(n')} \|_{L_{t,x}^{4}((\inf(I^{(n')}), 0] \times \mathbb{R}^{2})} = \| u^{(n')} \|_{L_{t,x}^{4}([0, \sup(I^{(n')})) \times \mathbb{R}^{2})} = \infty.
\end{equation}
Indeed, if $(\ref{4.9})$ is true, then following the argument in $(\ref{3.16})$--$(\ref{3.18})$, either $t_{n}^{1} \nearrow \infty$ or $t_{n}^{1} \searrow -\infty$ would contradict $(\ref{4.9})$. 

The fact that $(\ref{4.9})$ holds follows from $(\ref{4.2})$ and $t_{n'(n)} \nearrow \sup(I)$, and applying a profile decomposition to $u(t_{n'(n)})$. In that case as well, $(\ref{4.8})$ must hold, and furthermore, if $t_{n}^{1} \searrow -\infty$, then we contradict
\begin{equation}\label{4.10}
\lim_{n \rightarrow \infty} \| u \|_{L_{t,x}^{4}((\inf(I), t_{n}] \times \mathbb{R}^{2})} = \infty,
\end{equation}
and if $t_{n}^{1} \nearrow \infty$, then we contradict
\begin{equation}\label{4.11}
\lim_{n \rightarrow \infty} \| u \|_{L_{t,x}^{4}([t_{n}, \sup(I)) \times \mathbb{R}^{2})} = \infty.
\end{equation}
\begin{remark}
The fact that $(\ref{4.9})$ holds is heavily reliant on the fact that $\| u \|_{L^{2}}^{2} < 16 \pi (m + 1)$, and therefore, at most one profile can blow up. If $\| u \|_{L_{2}}^{2} \geq 16 \pi (m + 1)$, then a solution could conceivably decouple into a profile that blows up forward in time and scatters backward in time, and a profile that blows up backward in time and scatters forward in time.
\end{remark}

Furthermore, we must have
\begin{equation}\label{4.12}
\| \phi^{j} \|_{L^{2}} = 0, \qquad \text{for any} \qquad j \geq 2, \qquad \text{and} \qquad \lim_{n' \rightarrow \infty} \| w_{n'} \|_{L^{2}} = 0.
\end{equation}
Otherwise, by $(\ref{3.6})$, $\| \phi^{1} \|_{L^{2}}^{2} < m$, which contradicts $(\ref{4.5})$. Indeed, $(\ref{4.7})$ implies
\begin{equation}\label{4.13}
\lambda_{n'} u_{0}^{(n')} (\lambda_{n'} \cdot) \rightharpoonup \phi^{1},
\end{equation}
weakly in $L^{2}$. Thus, for any $f \in L_{m}^{2}$,
\begin{equation}\label{4.14}
\langle \lambda_{n'} u_{0}^{(n')}(\lambda_{n'} \cdot), f \rangle \rightarrow \langle \phi^{1}, f \rangle,
\end{equation}
where $\langle \cdot, \cdot \rangle$ is the standard $L^{2}$ inner product. Furthermore, by $(\ref{4.6})$, for any $n'$,
\begin{equation}\label{4.15}
\langle \lambda_{n'} \lambda(t_{n(n')}) u(t_{n(n')}, \lambda_{n'} \lambda(t_{n(n')}) \cdot), f \rangle \rightarrow \langle \lambda_{n'} u_{0}^{(n')}(\lambda_{n'} \cdot), f \rangle.
\end{equation}
Choosing $n(n') \geq N(f, n')$ sufficiently large,
\begin{equation}\label{4.16}
| \langle \lambda_{n'} \lambda(t_{n(n')}) u(t_{n(n')}, \lambda_{n'} \lambda(t_{n(n')}) \cdot), f \rangle - \langle \lambda_{n'} u_{0}^{(n')}(\lambda_{n'} \cdot), f \rangle| < 2^{-n'},
\end{equation}
\begin{equation}\label{4.17}
\langle \lambda_{n'} \lambda(t_{n(n')}) u(t_{n(n')}, \lambda_{n'} \lambda(t_{n(n')}) \cdot), f \rangle \rightarrow \langle \phi^{1}, f \rangle.
\end{equation}
Since $L_{m}^{2}$ is separable, we can make a diagonal argument proving that $\| \phi^{1} \|_{L^{2}}^{2} \in \mathcal M_{u}$, which gives contradiction. Indeed, let $f_{1}, f_{2}, f_{3} ...$ be a countable basis of $L_{m}^{2}$. Then, for any $n'$, take
\begin{equation}\label{4.18}
n(n') \geq \max \{ N(f_{1}, n'), ..., N(f_{n'}, n') \},
\end{equation}
such that $|\sup(I) - t_{n(n')}| < \frac{1}{2} |\sup(I) - t_{n(n' - 1)}|$. Then, $(\ref{4.17})$ implies
\begin{equation}\label{4.19}
\lambda_{n'} \lambda(t_{n(n')}) u(t_{n(n')}, \lambda_{n'} \lambda(t_{n(n')}) \cdot) \rightharpoonup \phi^{1}.
\end{equation}
\end{proof}

Now suppose $\tilde{u}_{0} \in L_{m}^{2}$ is such that
\begin{equation}\label{4.20}
\| \tilde{u}_{0} \|_{L^{2}}^{2} = \inf \mathcal M_{u},
\end{equation}
and let $\tilde{u}$ be the solution to $(\ref{1.1})$ with initial data $\tilde{u}_{0}$.
\begin{lemma}\label{l4.2}
The solution $\tilde{u}$ is almost periodic.
\end{lemma}
\begin{proof}
We already know that 
\begin{equation}\label{4.21}
\| \tilde{u} \|_{L_{t,x}^{4}(\inf(\tilde{I}), 0] \times \mathbb{R}^{2})} = \| \tilde{u} \|_{L_{t,x}^{4}([0, \sup(\tilde{I})) \times \mathbb{R}^{2})} = \infty.
\end{equation}
Let $\tilde{t}_{n}$ be a sequence of times in $\tilde{I}$. We know that
\begin{equation}\label{4.22}
\lambda(t_{n'}) u(t_{n'}, \lambda(t_{n'}) \cdot) \rightharpoonup \tilde{u}_{0},
\end{equation}
and for any $\tilde{t}_{n} \in \tilde{I}$,
\begin{equation}\label{4.23}
\lambda(t_{n'}) u(t_{n'} + \lambda(t_{n'})^{2} \tilde{t}_{n}, \lambda(t_{n'}) \cdot) \rightharpoonup \tilde{u}(\tilde{t}_{n}).
\end{equation}

Repeating the analysis in Lemma $\ref{l4.1}$ forces a single profile with mass equal to the infimum of $\mathcal M_{u}$, obtaining almost periodicity.
\end{proof}

Combining Theorem $\ref{t2.4}$ with Lemma $\ref{l4.2}$ and the proof of Lemma $\ref{l4.1}$ implies Theorem $\ref{t1.1}$.

\section{Acknowledgements}
During the writing of this paper, the author was partially supported by NSF grant DMS-2153750.

\bibliography{biblio}

\begin{thebibliography}{KKO22}

\bibitem[Dod19]{dodson2019defocusing}
Benjamin Dodson.
\newblock {\em Defocusing {N}onlinear {S}chr{\"o}dinger {E}quations}, volume
  217.
\newblock Cambridge University Press, 2019.

\bibitem[Dod21]{dodson20212}
Benjamin Dodson.
\newblock The ${L}^{2}$ sequential convergence of a solution to the
  mass-critical {NLS} above the ground state.
\newblock {\em arXiv preprint arXiv:2101.09172}, 2021.

\bibitem[Dod22]{dodson2022l2}
Benjamin Dodson.
\newblock The ${L}^{2}$ sequential convergence of a solution to the
  mass-critical {NLS} above the ground state.
\newblock {\em Nonlinear Analysis}, 215:112612, 2022.

\bibitem[Dod23]{dodson2023liouville}
Benjamin Dodson.
\newblock A {L}iouville theorem for the {C}hern--{S}imons--{S}chr{\"o}dinger
  equation.
\newblock {\em arXiv preprint arXiv:2302.12384}, 2023.

\bibitem[Fan21]{fan20182}
Chenjie Fan.
\newblock The ${L}^{2}$ {W}eak {S}equential {C}onvergence of {R}adial
  {F}ocusing {M}ass {C}ritical {NLS} {S}olutions with {M}ass {A}bove the
  {G}round {S}tate.
\newblock {\em Int. Math. Res. Not. IMRN}, (7):4864--4906, 2021.

\bibitem[KK19]{kim2019pseudoconformal}
Kihyun Kim and Soonsik Kwon.
\newblock On pseudoconformal blow-up solutions to the self-dual
  {C}hern-{S}imons-{S}chr{\"o}dinger equation: existence, uniqueness, and
  instability.
\newblock {\em arXiv preprint arXiv:1909.01055}, 2019.

\bibitem[KK20]{kim2020construction}
Kihyun Kim and Soonsik Kwon.
\newblock Construction of blow-up manifolds to the equivariant self-dual
  {C}hern-{S}imons-{S}chr{\"o}dinger equation.
\newblock {\em arXiv preprint arXiv:2009.02943}, 2020.

\bibitem[KKO20]{kim2020blow}
Kihyun Kim, Soonsik Kwon, and Sung-Jin Oh.
\newblock Blow-up dynamics for smooth finite energy radial data solutions to
  the self-dual {C}hern-{S}imons-{S}chr{\"o}dinger equation.
\newblock {\em arXiv preprint arXiv:2010.03252}, 2020.

\bibitem[KKO22]{kim2022soliton}
Kihyun Kim, Soonsik Kwon, and Sung-Jin Oh.
\newblock Soliton resolution for equivariant self-dual
  {C}hern-{S}imons-{S}chr{\"o}dinger equation in weighted {S}obolev class.
\newblock {\em arXiv preprint arXiv:2202.07314}, 2022.

\bibitem[KTV09]{killip2009cubic}
Rowan Killip, Terence Tao, and Monica Vișan.
\newblock The cubic nonlinear schr{\"o}dinger equation in two dimensions with
  radial data.
\newblock {\em Journal of the European Mathematical Society}, 11(6):1203--1258,
  2009.

\bibitem[LL22]{li2022threshold}
Zexing Li and Baoping Liu.
\newblock On threshold solutions of the equivariant
  {C}hern--{S}imons--{S}chr{\"o}dinger equation.
\newblock {\em Annales de l'Institut Henri Poincar{\'e} C}, 39(2):371--417,
  2022.

\bibitem[LS16]{liu2016global}
Baoping Liu and Paul Smith.
\newblock Global wellposedness of the equivariant
  {C}hern--{S}imons--{S}chr{\"o}dinger equation.
\newblock {\em Revista Matem{\'a}tica Iberoamericana}, 32(3):751--794, 2016.

\bibitem[Mer92]{merle1992uniqueness}
Frank Merle.
\newblock On uniqueness and continuation properties after blow-up time of
  self-similar solutions of nonlinear {S}chr{\"o}dinger equation with critical
  exponent and critical mass.
\newblock {\em Communications on pure and applied mathematics}, 45(2):203--254,
  1992.

\bibitem[Mer93]{merle1993determination}
Frank Merle.
\newblock Determination of blow-up solutions with minimal mass for nonlinear
  {S}chr{\"o}dinger equations with critical power.
\newblock {\em Duke Mathematical Journal}, 69(2):427--454, 1993.

\bibitem[Str77]{strichartz1977restrictions}
Robert~S Strichartz.
\newblock Restrictions of {F}ourier transforms to quadratic surfaces and decay
  of solutions of wave equations.
\newblock {\em Duke Mathematical Journal}, 44(3):705--714, 1977.

\bibitem[TVZ08]{MR2445122}
Terence Tao, Monica Visan, and Xiaoyi Zhang.
\newblock Minimal-mass blowup solutions of the mass-critical {NLS}.
\newblock {\em Forum Math.}, 20(5):881--919, 2008.

\end{thebibliography}
\bibliographystyle{alpha}
\end{document}